\documentclass[reqno,centertags,12pt]{amsart}
\usepackage[letterpaper,margin=1.3in]{geometry}
\usepackage{amsmath,amsthm,amsfonts,amssymb,enumerate}
\usepackage[bookmarksopen=false,final]{hyperref}
\usepackage[utf8]{inputenc}

\newcommand{\no}{\notag}

\newcommand{\bb}{\mathbb}

\newcommand{\ol}{\overline}

\newcommand{\ti}{\tilde}

\newcommand{\bs}{\backslash}
\newcommand{\al}{\alpha}
\newcommand{\be}{\beta}

\newcommand{\Om}{\Omega}

\newcommand{\te}{\theta}

\newcommand{\eps}{\varepsilon}
\newcommand{\pd}{\partial}
\newcommand{\supp}{\mathrm{supp}}

\newcommand{\dist}{\mathrm{dist}}
\newcommand{\ca}{\mathrm{Cap}}
\newcommand{\x}{{\mathbf x}}
\newcommand{\cf}{{\mathbf1}}

\renewcommand{\Re}{\mathrm{Re}}
\renewcommand{\Im}{\mathrm{Im}}

\newtheorem{theorem}{Theorem}
\newtheorem{lemma}[theorem]{Lemma}
\newtheorem{corollary}[theorem]{Corollary}
\theoremstyle{definition}
\newtheorem{definition}[theorem]{Definition}

\newtheorem{remark}[theorem]{Remark}
\newtheorem{conjecture}[theorem]{Conjecture}
\numberwithin{equation}{section}
\numberwithin{theorem}{section}
\begin{document}
	\title{Lower bounds for weighted Chebyshev and Orthogonal polynomials}
	
	\author{G\"{o}kalp Alpan}
	\address{Faculty of engineering and natural sciences,
		Sabancı University, İstanbul, Turkey}
	\email{gokalp.alpan@sabanciuniv.edu}
	\thanks{\footnotesize G.A. is supported by the Scientific and Technological Research Council of Türkiye (TÜBİTAK) ARDEB 1001 Grant Number 123F358. G.A is also supported by the BAGEP Award of the Science Academy.}
	
	\author{Maxim Zinchenko}
	\address{Department of Mathematics and Statistics, University of New Mexico, Albuquerque, NM 87131, USA}
	\email{maxim@math.unm.edu}
	\thanks{\footnotesize M.Z. is supported in part by the Simons Foundation Grant MP--TSM--00002651.}
	
	
	\begin{abstract}
        We derive optimal asymptotic and non-asymptotic lower bounds on the Widom factors for weighted Chebyshev and orthogonal polynomials on compact subsets of the real line. In the Chebyshev case we extend the optimal non-asymptotic lower bound previously known only in a handful of examples to regular compact sets and a large class weights. Using the non-asymptotic lower bound, we extend Widom's asymptotic lower bound for weights bounded away from zero to a large class of weights with zeros including weights with strong zeros and infinitely many zeros. As an application of the asymptotic lower bound we extend Bernstein's 1931 asymptotics result for weighted Chebyshev polynomials on an interval to arbitrary Riemann integrable weights with finitely many zeros and to some continuous weights with infinitely many zeros. In the case of orthogonal polynomials, we derive optimal asymptotic and non-asymptotic lower bound on arbitrary regular compact sets for a large class of weights in the non-asymptotic case and for arbitrary Szeg\H{o} class weights in the asymptotic case, extending previously known bounds on finite gap and Parreau--Widom sets.
	\end{abstract}
	
	\date{\today}
	\maketitle
	%%%%%%%%%%%%%%%%%%%%%%%%%%%%%%%%%%%%%%%%%%%%%%%%%%%%%%%%%%%%%%%%%%%%%
	
	%%%%%%%%%%%%%%%%%%%%%%%%%%%%%%%%%%%%%%%%%%%%%%%%%%%%%%%%%%%%%%%%%%%%%%%
	\section{Introduction}
	%%%%%%%%%%%%%%%%%%%%%%%%%%%%%%%%%%%%%%%%%%%%%%%%%%%%%%%%%%%%%%%%%%%%%%%
	Let $K$ be a non-polar compact subset of $\bb C$ and denote the logarithmic capacity of $K$ by $\ca(K)$. We have $\ca(K)>0$ in this case and so there exists the equilibrium measure of $K$ which we denote by $\mu_K$. Let $\Om_K$ be the unbounded component of $\overline{\bb C}\bs K$, then the Green function for the domain $\Om_K$ is given by
	\begin{align}\label{GrFn}
		g_K(z) = -\log\ca(K)+\int\log|z-\zeta|\,d\mu_K(\zeta), \quad z\in\bb C.
	\end{align}
	We note that $g_K$  has the following asymptotics at infinity
	\begin{align}\label{GrAs}
		g_K(z) = \log|z| - \log\ca(K) + o(1).
	\end{align}
	We use $\| \cdot \|_K$ to denote the sup-norm on $K$. We call $w:K\to[0,\infty)$ a weight function on $K$ if $w$ is Borel measurable and takes positive values on a subset of $K$ containing infinitely many points. If $w$ is a bounded weight function then the $n$-th Chebyshev polynomial with respect to $w$ on $K$, denoted by $T_{n,w}$, is the unique minimizer of $\|wP_n\|_K$ over all monic polynomials $P_n$ of degree $n$. We refer the reader to \cite{NSZ21} for a recent review of  weighted Chebyshev polynomials. Let
	\begin{align}
		t_n(K,w):= \|wT_{n,w}\|_K
	\end{align}
and define the $n$-th \emph{Widom factor for the sup-norm} with respect to a bounded weight function $w$ on $K$ by
	\begin{align}
		W_{\infty,n}(K,w):= \frac{t_n(K,w)}{\ca(K)^n}.
	\end{align}
In addition, we define the Szeg\H{o} factor of $w$ on $K$ by
	\begin{align}\label{S-def}
		S(K,w):=\exp\biggl[{\int_K \log{w}\,d\mu_K}\bigg].
	\end{align}
By Jensen's inequality, $S(K,w) \le \int w d\mu_K$ and hence the Szeg\H{o} factor is always finite for bounded (and more generally integrable with respect to $d\mu_K$) weights. We say that $w$ is in the \emph{Szeg\H{o} class} on $K$ if $S(K,w)>0$. For example, every weight bounded away from zero (i.e., $w(x)\ge c>0$ for $x\in K$) is in the Szeg\H{o} class on $K$.
	
    \smallskip

	We refer to the case of constant weight $w=1$ on $K$ as the unweighted case. It is trivial to see that $S(1)=1$.
	When $K\subset \mathbb{C}$ (see \cite[Theorem~5.5.4]{Ran95}),
	\begin{equation}
		W_{\infty,n}(K,1)\geq 1,\,\,\, n\in\bb N,
	\end{equation}
	and when $K\subset \bb R$ (see \cite{Sch08})
	\begin{equation}\label{Schief-ineq}
		W_{\infty,n}(K,1)\geq 2,\,\,\, n\in\bb N.
	\end{equation}
	In the weighted case the corresponding lower bound has the following form (see \cite[Theorem~13]{NSZ21} and also \cite[Chapter~I, Theorem~3.6]{ST97}),
	\begin{equation}\label{cheb univ}
		W_{\infty,n}(K,w)\geq S(K,w),\,\,\, n\in\bb N.
	\end{equation}
	In general, the lower bound \eqref{cheb univ} cannot be improved. Indeed, if $K\subset \bb R$ then for each $n\in \bb N$,
	\begin{equation}
		\inf_w \frac{W_{\infty,n}(K,w)}{S(K,w)}=1,
	\end{equation}
	where the infimum is taken over all polynomials $w$ positive on $K$ (see \cite[Theorem~13]{NSZ21}).
	However, several examples of weights are known for which the following doubled lower bound holds,
	\begin{equation}\label{double s cheb}
		W_{\infty,n}(K,w)\geq 2 S(K,w),\,\,\, n\in\bb N.
	\end{equation}
	The constant weight gives one such example (cf.\ \eqref{Schief-ineq}). In addition, if $-1,1\in K$ and $K\subset [-1,1]$ then \eqref{double s cheb} holds for $w(x)=\sqrt{1-x^2}$ (see \cite{SZ21}) and for $w(x)=\sqrt{1-x}$ and $w(x)=\sqrt{1+x}$ (see \cite{Alp22}). In Theorem~\ref{LB-Thm}, we show that \eqref{double s cheb} holds for weights that are square roots of rational functions under an additional explicit condition on the zeros. We refer the reader to \cite[Lemma~2.12]{ELY24} for a recent relevant result. In Theorems~\ref{prod-w-thm} and \ref{prod-w-thm2}, we prove that \eqref{double s cheb} also holds for weights that are certain infinite products of rational functions.
	
    \smallskip

	In \cite[Theorem~11.5]{Wid69}, Widom proved the following asymptotic lower bound on a finite gap set $K\subset\bb R$,
	\begin{align}\label{Wid-ALB-Cheb}
		\liminf_{n\to\infty} W_{\infty,n}(K,w)\geq 2 S(K,w),
	\end{align}
	for weights $w$ that are Riemann integrable and bounded away from zero on $K$.
	In Theorem~\ref{ALB-Cheb-thm}, we extend this result to weights of the form $w = w_0w_1$, where $w_1$ is Riemann integrable and bounded away from zero on $K$ and $w_0$ is an arbitrary continuous weight with at most finitely many
zeros on $K$.
	
    \smallskip

	In the case of an interval $K=[-1,1]$, asymptotics of the Widom factors was first studied by Bernstein in \cite{Ber31}. Let
	\begin{equation}\label{w prod}
		w(x)= w_0(x)w_1(x) \mbox{ on } K,
	\end{equation}
	where $w_1$ is a Riemann integrable weight function bounded away from $0$ on $K$ and
	\begin{equation}
		w_0(x)=\prod_{j=1}^d |x-\alpha_j|^{\beta_j},
	\end{equation}
	where $d\in \bb N$ and $\alpha_j\in K$, $\beta_j>0$ for all $j$.
	Then, $w$ is in the Szeg\H{o} class on $K$ and (see \cite{Ber31} and also \cite[Theorem~5]{CER24})
	\begin{equation}\label{doubled limit cheb}
		\lim_{n\rightarrow\infty} W_{\infty,n}(K,w)=2 S(K,w).
	\end{equation}

	In \cite{LS87},  Lubinsky--Saff studied asymptotics of $L^p$ Widom factors on $K=[-1,1]$. It could be deduced from their work and the work of Widom \cite{Wid69} that Bernstein's asymptotics \eqref{doubled limit cheb} extends to upper semi-continuous weight functions $w$ satisfying $1/w \in L^p(K,dx)$ for all $p<\infty$. While this extension allows only mild zeros, it nevertheless suggests that the assumption of Riemann integrability in Bernstein's result is likely unnecessary for \eqref{doubled limit cheb}.
	
	In Theorem~\ref{Bern-extension}, we extend Bernstein's asymptotics result to weights with more general zeros. In particular, we show that \eqref{doubled limit cheb} holds for $w=w_0 w_1$, where $w_1$ is as above and $w_0$ is an arbitrary continuous weight with at most finitely many zeros on $K$. We pose no restrictions of the type of zeros and, in particular, we allow Szeg\H{o} class weights with ``strong zeros" such as $w(x)=\exp(-1/|x|^\alpha)$, $0<\alpha<1$, and also non-Szeg\H{o} class weights. In addition, we show that \eqref{doubled limit cheb} holds for certain continuous weights of the Szeg\H{o} class with infinitely many zeros on $K$. %weights that are square root of certain rational functions, infinite products of rational functions.
	
    \smallskip

	For a non-polar compact set $K$ in $\bb C$, let $d\mu=w\, d\mu_K$ where $w$ is a weight function on $K$ with $w\in L^1(K,d\mu_K)\bs\{0\}$. Let $P_n$ denote the $n$-th monic orthogonal polynomial associated with $\mu$. We define the $n$-th \emph{$L^2$ Widom factor} for $w$ on $K$ by
	\begin{equation}\label{l2 wid}
		W_{2,n}(K,w):=\frac{\|P_n\|_{L^2(\mu)}}{\ca(K)^n}.
	\end{equation}
	Then for $K\subset \bb C$ (see \cite{Alp19}, \cite{AZ20a}),
	\begin{equation}
		[W_{2,n}(K,1)]^2\geq 1, \,\,\, n\in \bb N,
	\end{equation}
	and for $K\subset \bb R$ (see \cite{AZ20a}),
	\begin{equation}
		[W_{2,n}(K,1)]^2\geq 2, \,\,\, n\in \bb N.
	\end{equation}
	Just like the $L^\infty$ case, there is a universal lower bound when $K\subset \bb C$ (see \cite{Alp19}, \cite{AZ20a}):
	\begin{equation}\label{good old uni}
		[W_{2,n}(K,w)]^2\geq S(K,w), \,\,\, n\in \bb N.
	\end{equation}
	The inequality \eqref{good old uni} cannot be improved in the following sense (see \cite[Theorem~2.2]{AZ20b}):
	If $K\subset \bb R$ then for each $n\in \bb N$,
	\begin{equation}
		\inf_w \frac{[W_{2,n}(K,w)]^2}{S(K,w)}=1,
	\end{equation}
	where the infimum is taken over all polynomials $w$ positive on $K$.
	However, several families of weights have been discovered with the improved lower bound
	\begin{equation}\label{improv l2}
		[W_{2,n}(K,w)]^2\geq 2 S(K,w), \,\,\, n\in \bb N.
	\end{equation}
	The examples with \eqref{improv l2} include the unweighted case (see \cite{AZ20a}), weights defined in terms isospectral torus (see \cite{AZ20a}), Jacobi weights for certain parameters (see \cite{AZ20a}) and generalized Jacobi measures (see \cite{Alp22}). We prove in Theorems~\ref{ALB-OP-Thm} and \ref{LB-OP-Thm} that \eqref{improv l2} holds for certain rational weights and infinite product of rational functions.
	
    \smallskip

	When $K$ is a finite gap set and $w$ is an arbitrary weight function in $L^1(K,d\mu_K)$, it follows from \cite[Theorem~12.3]{Wid69} that
	\begin{equation}\label{univ liminf l2}
		\liminf_{n\rightarrow\infty}[W_{2,n}(K,w)]^2\geq 2 S(K,w).
	\end{equation}
	There are also asymptotics results for $[W_{2,n}(K,w)]^2$ beyond finite gap sets \cite{Chr12}, \cite{PehYud03}.
	
	In Theorem~\ref{ALB-OP-Thm}, we show that for any non-polar compact $K\subset \bb R$ and arbitrary weight $w$, \eqref{univ liminf l2} holds. This is the best general result we can get as the inequality in \eqref{univ liminf l2} becomes equality when $K$ is an interval in view of Szeg\H{o}'s theorem, see \cite[Theorem~13.8.8]{Sim05}. The equality  holds in \eqref{univ liminf l2} if $K\subset \bb R$ is an inverse polynomial image of $[-1,1]$ and $w:=1$, see \cite[Theorem~4.4]{AZ20b}.
	
	Both $\liminf_{n\rightarrow\infty} W_{\infty,n}(K,1)$ and $\liminf_{n\rightarrow\infty}[W_{2,n}(K,1)]^2$ can be made arbitrarily large or even $\infty$, see \cite{GonHat15}, \cite{AlpGon16} and also Remark \ref{kgamma}. Hence the exact values for these limits are set dependent and might well exceed $2S(K,w)$.
	
    \smallskip

	The inequality \eqref{univ liminf l2} can be used to obtain a better lower bound than what was previously known for
	$\liminf_{n\rightarrow\infty} W_{\infty,n}(K,w)$ in the case of $K\subset \bb R$ and arbitrary bounded Szeg\H{o} class weight $w$. We have
	\begin{equation}
		\left(\frac{\|T_{n,w} w\|_{L^\infty(K)}}{\ca{K}^n}\right)^2\geq \frac{\int T_{n,w}^2 w^2 d\mu_K}{\ca(K)^{2n}}.
	\end{equation}
	Taking the $\liminf$ and using \eqref{univ liminf l2} we get
	\begin{equation}
		\liminf_{n\rightarrow\infty }[W_{\infty,n}(K,w)]^2\geq 2 S(K,w^2)
	\end{equation}
	and thus
	\begin{equation}\label{not that sharp}
		\liminf_{n\rightarrow\infty}[W_{\infty,n}(K,w)]\geq \sqrt{2} S(K,w),
	\end{equation}
	where we get an extra $\sqrt{2}$ term compared to what we would get if we take the liminf of both sides in \eqref{cheb univ}.
	
	It is quite likely that \eqref{not that sharp} does not give the optimal lower bound. The lower bound is a result of a simple trick to compare the norms of $L^2$ and $L^\infty$ extremal polynomials. On the other hand, in the well studied cases such as the case of an interval or a finite gap set $K$ with a sufficiently nice weight $w$ (see \cite{Wid69}) or in the unweighted case for more general compacts sets $K\subset\bb R$ (see \cite[Theorem~5.2]{CSZ5}), $2S(K,w)$ appears as the value of $\liminf_{n\rightarrow\infty}W_{\infty,n}(K,w)$. We also have several other examples in this article with the same property. Hence, we make the following conjecture:
	\begin{conjecture}
		Let $K$ be a non-polar compact subset of $\bb R$ and $w$ be an arbitrary bounded weight function on $K$. Then,
%in the Szeg\H{o} class of $K$.
		\begin{equation}
			\liminf_{n\rightarrow\infty} W_{\infty,n}(K,w)\geq 2 S(K,w).
		\end{equation}
	\end{conjecture}
%The conjectured asymptotic lower bound is not known even in the case of an interval $K=[-1,1]$ for general bounded weights $w$.

	The plan of the paper is as follows. In Section 2, we discuss non-asymptotic lower bounds  and in Section 3, we discuss asymptotic lower bounds for $W_{\infty,n}(K,w)$. In Section 4, we prove a sharp asymptotic lower bound for $[W_{2,n}(K,w)]^2$ on $\bb R$ and also show that the improved non-asymptotic lower bound \eqref{improv l2} holds for some rational functions and infinite product of rational functions.
	
	%\begin{theorem}[{\cite[Theorem 5.2.5]{Ran95}}]
	%Let $q(z)=\sum_{j=0}^d a_j z^j$ be a polynomial of degree $d$ with $d\geq 1$. Let $K$ be a compact subset of $\bb C$. Then
	%\begin{align}
	%\ca{(q^{-1}(K))}=\left(\frac{\ca{(K)}}{|a_d|}\right)^{1/d}.
	%\end{align}
	%\end{theorem}

	%%%%%%%%%%%%%%%%%%%%%%%%%%%%%%%%%%%%%%%%%%%%%%%%%%%%%%%%%%%%%%%%%%%%%%%
	\section{Lower Bounds for $L^\infty$ norm}
	%%%%%%%%%%%%%%%%%%%%%%%%%%%%%%%%%%%%%%%%%%%%%%%%%%%%%%%%%%%%%%%%%%%%%%%
	
	In this section, we derive an $n$-independent lower bound on $W_{\infty,n}(K,w)$ for arbitrary regular compact sets $K\subset\bb R$ and weights of the form $w(x)=\sqrt{R(x)}$ where $R$ is a rational function bounded and non-negative on $K$. Earlier work on lower bounds for weighted Chebyshev polynomials include \cite{SZ21, NSZ21, Alp22}. In \cite[Appendix~A]{Ach56}, Akhiezer gave an explicit formula for weighted Chebyshev polynomials on $[-1,1]$ for the weights $w=1/\sqrt{P}$, where $P$ is a strictly positive polynomial on $[-1,1]$.
	
	We start by extending the well known relation between the capacity of a compact set and the capacity of its polynomial preimage set \cite[Theorem~5.2.5]{Ran95} to the case of a rational function preimage. %In the following we will assume that rational functions are defined at all points of the complex plane and take values in $\ol{\bb C}$.
	
	\begin{lemma}\label{CapPreImg}
		Let $K\subset\bb C$ be a regular compact set and let
		\begin{align}
			R(z) = c\frac{\prod_{j=1}^{d_0}(z-a_j)} {\prod_{j=1}^{d_1}(z-b_j)}
		\end{align}
		be a rational function with a pole at infinity of order $n=d_0-d_1\geq1$ and assume that the numerator and denominator of $R$ have no common zeros. Then the preimage set $L=R^{-1}(K)=\{z\in\bb C:R(z)\in K\}$
		is a regular compact set and
		\begin{align}\label{capKcapL}
			\ca(K) = |c|\,\ca(L)^n\exp\biggl[-\sum_{j=1}^{d_1}g_L(b_j)\bigg].
		\end{align}
	\end{lemma}
	\begin{proof}
		By assumption $R(\infty)=\infty$ and $K$ is compact hence the preimage set $L$ is compact.
		Let $\Om_K$ and $\Om_L$ be the unbounded components of $\ol{\bb C}\bs K$ and $\ol{\bb C}\bs L$, respectively. Then $R(\Om_L)=\Om_K$ and $R(\pd\Om_L)=R(\pd\Om_K)$. Let $g_K$ and $g_L$ be the Green functions on the domains $\Om_K$ and $\Om_L$, respectively, and consider the function
		\begin{align}
			h(z) = g_K(R(z)) - ng_L(z) - \sum_{j=1}^{d_1} g_L(z,b_j).
		\end{align}
		Note that all the logarithmic poles at $\infty,b_1,\dots,b_{d_1}$ cancel, so $h$ extends to a harmonic function on $\Om_L$. Since $h$ has zero boundary values q.e.\ on $\pd\Om_L$, it follows from the maximum principle that $h\equiv0$ on $\Om_L$. By assumption, $K$ is regular, so $g_K(R(z))\to0$ as $z\to z_0$ for each $z_0\in\pd\Om_L$. Since $g_L(z)$ and $g_L(z,b_j)$ for all $j$ are non-negative on $\Om_L$ it follows that $g_L(z)\to 0$ as $z\to z_0$ for each $z_0\in\pd\Om_L$, that is, $L$ is regular.
		In addition, we have
		\begin{align}
			0=h(\infty)=\log|c|-\log\ca(K)+n\log\ca(L)-\sum_{j=1}^{d_1}g_L(b_j)
		\end{align}
		which implies \eqref{capKcapL}.
	\end{proof}
	
	We will also need the following variant of \cite[Lemma~13]{SZ21} for compact subsets of the complex plane.
	
	\begin{lemma}\label{Lem:K=L}
		Let $K\subset L$ be two compact subsets of $\bb C$ such that $L$ is regular, $K$ is polynomially convex, and $\ca(K)=\ca(L)$. Then $K=L$.
	\end{lemma}
	\begin{proof}
		We only need to show $L\subset K$. Since $L$ is regular we have $\ca(L)>0$. Then, by \cite[Proposition~2.1]{CSZ3}, $\supp(\mu_L) \subset K$, where $\mu_L$ denotes the equilibrium measure of $L$. Let $\Om$ be the unbounded component of $\bb C\bs\supp(\mu_L)$ and $\Om_K^\prime$ be the unbounded component of $\bb C\bs K$. Then $\Om\supset \Om_K^\prime$. Since $K$ is polynomially convex, we have $K=\bb C\bs\Om_K^\prime$ and hence $\bb C\bs\Om\subset K$.
		It follows from \eqref{GrFn} that $g_L(z)$ is a harmonic function on $\Om$. Since $g_L$ is non-negative and non-constant on $\Om$, it must be strictly positive on $\Om$ by the minimum principle for harmonic functions. Then since $L$ is regular, we have $g_L=0$ on $L$ and hence $L\subset\bb C\bs\Om$. Thus, $L\subset K$.
	\end{proof}
	
	Our first main result of this section is the following non-asymptotic lower bound.
	
	\begin{theorem}\label{LB-Thm}
		Let $K\subset\bb R$ be a compact non-polar set and let
		\begin{align}
			R(z) = c\frac{\prod_{j=1}^{d_0}(z-a_j)} {\prod_{j=1}^{d_1}(z-b_j)}
		\end{align}
		be a rational function such that $R$ is bounded, non-negative, and not identically zero on $K$. Define $d=d_1-d_0$ and $w=\sqrt{R}$ on $K$. Then, for all $n>d/2$,
		\begin{align}\label{LB1-Cheb}
			W_{\infty,n}(K,w)\geq 2S(K,w) \exp\biggl[ -\frac12\sum_{j=1}^{d}g_K(a_j) \bigg].
		\end{align}
		In addition, assume that the zeros of $R$ are regular points of $K$. Then for all $n>d/2$,
        \begin{align}\label{LB2-Cheb}
			W_{\infty,n}(K,w)\geq 2 S(K,w).
		\end{align}
        Furthermore, the equality in \eqref{LB2-Cheb} is attained for some $n>d/2$ if and only if there is a polynomial $Q_n$ of degree $n$ such that its zeros are disjoint from $\{b_j\}_{j=1}^{d_1}$ and
		\begin{equation}\label{rat inv1}
			K=(R Q_n^2)^{-1}([0,1]).
		\end{equation}
		In that case, $Q_n=r T_{n,w}$, where $r$ is a non-zero real number.
	\end{theorem}
	\begin{proof}
		Without loss of generality we may assume that the numerator and denominator of $R$ have no common zeros. By assumption,
		\begin{align}
			w(x)= \sqrt{R(x)}= |c|^{1/2} \frac{\prod_{j=1}^{d_0}|x-a_j|^{1/2}} {\prod_{j=1}^{d_1}|x-b_j|^{1/2}},
		\end{align}
		hence using \eqref{S-def} and \eqref{GrFn} we obtain,
		\begin{align}\label{Sw}
			S(K,w) &= \exp{\biggl[ \int \log{w}\,d\mu_K \bigg]}
			\no \\
			&= |c|^{1/2} \exp{\biggl[
				\frac12\sum_{j=1}^{d_0} \int\log{|x-a_j|}\, d{\mu_K}(x) -
				\frac12\sum_{j=1}^{d_1} \int\log{|x-b_j|}\, d{\mu_K}(x) \bigg]}
			\no \\
			&=|c|^{1/2} \exp{\biggl[
				\frac12 \sum_{j=1}^{d_0} \big(g_K(a_j)+\log\ca(K)\big) -
				\frac12 \sum_{j=1}^{d_1} \big(g_K(b_j)+\log\ca(K)\big) \bigg]}
			\no \\
			&=|c|^{1/2} \ca(K)^{-d/2} \exp\biggl[
			\frac12 \sum_{j=1}^{d_0}g_K(a_j) -
			\frac12 \sum_{j=1}^{d_1}g_K(b_j) \bigg].
		\end{align}
		
		Next note that $(wT_{n,w})^2=RT_{n,w}^2$ is a rational function. We cancel any common zeros in the numerator and the denominator of $RT_{n,w}^2$ and assume without loss of generality that after the cancellation,
		%the degree of the polynomial in the denominator is $d_1^\prime$, where
		the remaining zeros of the denominator (repeated according to their multiplicities) are $\{b_j\}_{j=1}^{d_1'}$ where
		\begin{equation}\label{d1 prime}
			0\leq d_1^\prime \leq d_1.
		\end{equation}
		Define the rational preimage set
		\begin{align}\label{Ln}
			L_n &= \big(RT_{n,w}^2\big)^{-1}\big([0, t_n(K,w)^2]\big)
			\no\\
			&= \big\{ z\in\bb C: R(z)T_{n,w}^2(z) \in [0, t_n(K,w)^2] \big\}.
		\end{align}
		Since $\Re[T_{n,w}(x)]$ is a monic polynomial of degree $n$ and $|\Re[T_{n,w}(x)]| \le |T_{n,w}(x)|$ for all $x\in\bb R$, it follows that $T_{n,w}=\Re[T_{n,w}]$ so $T_{n,w}$ is a real polynomial. Since $w$ is also real on $K$ we have  $w(x)T_{n,w}(x)\in[-t_n(K,w),t_n(K,w)]$ for all $x\in K$. It follows that $K\subset L_n$ and hence
		\begin{align}\label{cap-g-ineq}
			\ca(K) \le \ca(L_n) \quad\text{and}\quad g_{L_n} \le g_K \text{ on } \bb C.
		\end{align}
		Since for $n > d/2$ the rational function $RT_{n,w}^2$ has a pole at infinity of order $2n-d$, it follows from Lemma~\ref{CapPreImg} that
		\begin{align}\label{some cap}
			\frac14t_n(K,w)^2 = \ca\big([0, t_n(K,w)^2]\big)
			= |c|\, \ca(L_n)^{2n-d} \exp\biggl[-\sum_{j=1}^{d_1'}g_{L_n}(b_j)\bigg].
		\end{align}
		Then using \eqref{d1 prime}, we obtain
		\begin{align}\label{tn2capL}
			t_n(K,w) &\ge 2 |c|^{1/2} \ca(L_n)^{n-d/2} \exp\biggl[-\frac12\sum_{j=1}^{d_1}g_{L_n}(b_j)\bigg],
		\end{align}
		where the inequality is strict if and only if some $b_j$ is a zero of $T_{n,w}$.
		Now it follows from \eqref{tn2capL}, \eqref{cap-g-ineq}, and \eqref{Sw} that for all $n>d/2$,
		\begin{align}\label{tn2capL2}
			W_{\infty,n}(K,w) = \frac{t_n(K,w)}{\ca(K)^{n}}
			&\geq 2|c|^{1/2} \ca(K)^{-d/2} \exp\biggl[-\frac12\sum_{j=1}^{d_1}g_K(b_j)\bigg]
			\no \\
			& = 2S(K,w) \exp\biggl[-\frac12\sum_{j=1}^{d_0}g_K(a_j)\bigg],
		\end{align}
		so the lower bound \eqref{LB1-Cheb} holds.

        When all the $a_j$'s are regular points of $K$, we have $g_K(a_j)=0$ for each $j$ and hence \eqref{LB2-Cheb} follows from \eqref{LB1-Cheb}.
		
		It remains to show that the equality in \eqref{LB2-Cheb} is equivalent to \eqref{rat inv1}. First assume that, $W_{\infty,n}(K,w)= 2 S(K,w).$ Then the inequality in \eqref{tn2capL2} becomes equality since by assumption $g_K(a_j)=0$ for all $j$. This is possible only if the inequalities in \eqref{tn2capL} and \eqref{cap-g-ineq} become equalities. Thus, the zeros of $T_{n,w}$ are disjoint from $\{b_j\}_{j=1}^{d_1}$ and $\ca(K)=\ca(L_n)$.
		Since $K\subset L_n$, $L_n$ is regular, and $K$ is polynomially convex due to $K\subset\bb R$, it follows from Lemma~\ref{Lem:K=L} that $K=L_n$.
		In view of \eqref{Ln}, choosing $Q_n=\frac{1}{t_n{(K,w)}}T_{n,w}$ then yields \eqref{rat inv1}.
		
		Conversely, assume \eqref{rat inv1} holds for some polynomial $Q_n$ of degree $n$ whose zeros do not include any element of $\{b_j\}_{j=1}^{d_1}$. Since $\sqrt{R}$ takes non-negative real values on $K$ and $K\subset \bb R$, it follows that $Q_n= r S_n$ where $r$ is a non-zero real number and $S_n$ is a monic polynomial with real coefficients. Hence $K=(R S_n^2)^{-1}([0,1/r^2])$. Then, by Lemma~\ref{CapPreImg}, we have
		\begin{equation}\label{iff ineq}
			\frac{1}{4r^2} = |c|\,\ca(K)^{2n-d}\exp\biggl[-\sum_{j=1}^{d_1}g_K(b_j)\bigg].
		\end{equation}
		Combining \eqref{iff ineq}, \eqref{tn2capL}, and \eqref{cap-g-ineq}, we get
		\begin{equation}
			\frac{1}{r^2}\leq t_n(K,w)^2.
		\end{equation}
		This implies that $\|w S_n\|_K \leq t_n(K,w)$. Since $S_n$ is monic, it follows from the uniqueness of the weighted Chebyshev polynomial that $S_n=T_{n,w}$. Then $K=L_n$ and since the zeros of $T_{n,w}=\frac1r Q_n$ do not include any elements of $\{b_j\}_{j=1}^{d_1}$, the inequality in \eqref{tn2capL} and hence in \eqref{tn2capL2} become equality. Thus, $W_{\infty,n}(K,w)= 2 S(K,w)$ holds.
	\end{proof}

	\begin{remark}
		\begin{enumerate}[(a)]
			%\item We always have $W_{\infty,n}(K,w)\geq S(K,w)$, see \cite[Theorem~13]{NSZ21}.
			
			\item The theorem applies to rational weights $w(x)=|R(x)|$ for any  complex rational function $R$ which is not identically zero since we can write $w(x)=[R(x)\ol{R}(x)]^{1/2}$, where $\ol{R}(x)$ denotes the rational function $R(x)$ with complex conjugated coefficients. Then using the symmetry $g_K(\bar z)=g_K(z)$ one obtains for all $n>d$,
			\begin{align}\label{LB3-Cheb}
				&W_{\infty,n}(K,w)\geq 2S(K,w) \exp\biggl[ -\sum_{j=1}^{d_0}g_K(a_j) \bigg],
				\\ \label{Sw-rat}
				&S(K,w) = |c|\, \ca(K)^{-d} \exp\biggl[ \sum_{j=1}^{d_0}g_K(a_j) - \sum_{j=1}^{d_1}g_K(b_j) \bigg],
			\end{align}
			where $\{a_j\}_{j=1}^{d_0}$ and $\{b_j\}_{j=1}^{d_1}$ are the zeros and poles of $R$ on $\bb C$, repeated according to their multiplicities, $d=d_1-d_0$, and $c$ is such that $R(z)\sim cz^{-d}$ at infinity.
			
			In the special case of rational functions $R$ that take real values on $K$, the lower bound \eqref{LB3-Cheb} can also be deduced from \eqref{Sw-rat} and the analog of Bernstein--Walsh lemma for real rational functions \cite[Lemma~2.12]{ELY24}.
			
			\item Instead of the square root weight, one can consider $m$-th root of a rational function, $w=\sqrt[m]{R}$, $m\in\bb N$. In this case, a similar proof yields for all $n>d/2$,
			\begin{align}\label{LB4-Cheb}
				W_{\infty,n}(K,w)
				&\geq A^{1/m}\, |c|^{1/m}\, \ca(K)^{-d/m} \exp\biggl[-\frac1m\sum_{j=1}^{d_1}g_K(b_j)\bigg]
				\no \\
				&= A^{1/m}\, S(K,w) \exp\biggl[ -\frac1m \sum_{j=1}^{d_0}g_K(a_j) \bigg],
			\end{align}
			where $A=2$ if $m$ is odd and $A=4$ if $m$ is even. Thus, our approach may lead to the optimal lower bound \eqref{LB2-Cheb} only in the cases $m=1$ and $m=2$.
			
			\item The theorem shows that the optimal lower bound \eqref{LB2-Cheb} holds for the Jacobi weights $w_{\al,\be}(x)=(1-x)^\al(1+x)^\be$ on $K=[-1,1]$ with $\al,\be\in\frac12\bb N_0$.
			
			A simple calculation shows that \eqref{LB2-Cheb} does not hold for the Jacobi weights $w_{\al,\al}$ with $\al\in(0,\frac12)$ and $n=1$ since
			\begin{align}
				W_{\infty,1}(K,w_{\al,\al}) \le \frac{\|xw_{\al,\al}\|_K}{\ca(K)} = \frac{2(2\al)^\al}{(1+2\al)^{\al+\frac12}} < \frac2{2^{2\al}} = 2S(K,w_{\al,\al}).
			\end{align}
			
			%\item Since $\|wP_n\|_K \ge \ca(K)^nW_{\infty,n}(K,w)$ for any monic polynomial $P_n$ of degree $n$, \eqref{LB1-Cheb}, \eqref{LB2-Cheb}, \eqref{LB3-Cheb}, \eqref{LB4-Cheb} provide lower bounds on $\|wP_n\|_K$ for arbitrary monic polynomials $P_n$ of degree $n$.
			
		\end{enumerate}
	\end{remark}	
	
	Our next goal is to show that the optimal lower bound \eqref{LB2-Cheb} extends to some weights with infinitely many zeros on $K$.
	As a preliminary step, we'll discuss uniform convergence of infinite products of rational functions.
	As is customary, we'll denote by $\prod_{j=1}^\infty f_j$ the limit of the partial products $\prod_{j=1}^k f_j$ if it exists.
	To start, we recall a well known result about convergence of infinite products.

	\begin{lemma}\label{prod-conv}
		Let $\{f_j(z)\}_{j=1}^\infty$ be a sequence of non-vanishing, continuous functions on a set $U$ such that $\sum_{j=1}^\infty(1-f_j(z))$ converges absolutely and locally uniformly on $U$ (i.e., on compact subsets of $U$). Then the infinite product $F(z)=\prod_{j=1}^\infty f_j(z)$ converges locally uniformly on $U$ to a finite, non-vanishing, continuous function $F(z)$ on $U$.
	\end{lemma}
	
	\begin{theorem}\label{prod-w-thm}
		Let $K\subset\bb R$ be a regular compact set and suppose $\{a_j\}_{j=1}^\infty\subset K$ and $\{b_j\}_{j=1}^\infty\subset\bb C\bs K$ are sequences such that
        $\sum_{j=1}^\infty g_K(b_j)<\infty$,
		$\sum_{j=1}^\infty |b_j-a_j|<\infty$,
		$\big\|\frac{x-a_j}{x-b_j}\big\|_K\le 1$ for every $j\ge1$, and the set $A=\{a_j\}_{j=1}^\infty$ is closed.
		Then the infinite product
		\begin{align}\label{prod-w}
			w(x) = \prod_{j=1}^\infty\Big|\frac{x-a_j}{x-b_j}\Big|
		\end{align}
		converges uniformly on $K$ to a weight function $w$ such that $w$ is continuous and bounded by $1$ on $K$, $w$ is positive on $K\bs A$, $S(K,w)>0$ (i.e., $w$ is in the Szeg\H{o} class on $K$), and for all $n\ge1$,
		\begin{align}\label{LB5-Cheb}
			W_{\infty,n}(K,w)\ge 2S(K,w).
		\end{align}
	\end{theorem}
	\begin{proof}
		Denote the partial products by  $w_k(x)=\prod_{j=1}^k\big|\frac{x-a_j}{x-b_j}\big|$, $k\in\bb N$. Then it follows from $\big\|\frac{x-a_j}{x-b_j}\big\|_K\le 1$ that the sequence $w_k$ is monotone decreasing as $k\to\infty$ and hence $w_k\to w$ pointwise on $K$. In addition, for each  $x_0\in A$ there is $k_0$ such that $x_0=a_{k_0}$. Then for all $k\ge k_0$ we have $0\le w(x)\le w_k(x) \le \big|\frac{x-x_0}{x-b_{k_0}}\big|$ on $K$ which implies that $w$ is continuous at $x_0$.
		
		Next, we note that, it follows from
		$\sum_{j=1}^\infty |b_j-a_j|<\infty$ that $|a_j-b_j|\to0$
		and hence the only limit points of the set $B=\{b_j\}_{j=1}^\infty$ are in $A$. Then locally uniformly on $K\bs A$ (i.e., for $x$ on compact subsets of $K\bs A$) we have $\dist(x,B)\ge c>0$ and it follows from
		\begin{align}
			0 \le 1-\Big|\frac{x-a_j}{x-b_j}\Big| \le \frac{|b_j-a_j|}{|x-b_j|} \le \frac{|b_j-a_j|}{\dist(x,B)}
		\end{align}
		that
		$\sum_{j=1}^\infty \bigl(1-\big|\frac{x-a_j}{x-b_j}\big|\big)$
		converges absolutely and locally uniformly on $K\bs A$. Since
		all the partial products $w_k$ are nonzero and continuous on $K\bs A$, it follows from Lemma~\ref{prod-conv} that the infinite product $w$ is also nonzero and continuous on $K\bs A$.
		
		Combining continuity of $w$ at every point of $A$ with continuity of $w$ on the (relatively) open set $K\bs A$, we conclude that $w$ is continuous on $K$. Since $w_k$ are decreasing to $w$ as $k\to\infty$, it follows from Dini's theorem that $w_k\to w$ uniformly on $K$. This implies that, for each fixed $n\ge1$,
		\begin{align}\label{wk-conv1}
			\lim_{k\to\infty} \|w_k T_{n,w}\|_K \to \|w T_{n,w}\|_K.
		\end{align}
		In addition, $w_k\searrow w$ and $w_k\le 1$ imply that $-\log w_k$ are non-negative and monotone increasing to $-\log w$ hence by the monotone convergence theorem,
		\begin{align}\label{wk-conv2}
			\lim_{k\to\infty} S(K,w_k) = S(K,w).
		\end{align}
		Since by \eqref{Sw-rat}, $S(K,w_k) = \exp\bigl[- \sum_{j=1}^k g_K(b_j)\big]$ for $k\ge1$, it follows from \eqref{wk-conv2} that $S(K,w) = \exp\bigl[-\sum_{j=1}^\infty g_K(b_j)\big]>0$.
		
		Finally, for each fixed $n\ge1$ we have by \eqref{LB3-Cheb},
		\begin{align}\label{wk-conv3}
			2S(K,w_k) \le \frac{\|w_k T_{n,w_k}\|_K}{\ca(K)^n} \le \frac{\|w_k T_{n,w}\|_K}{\ca(K)^n}.
		\end{align}
		Then taking limit as $k\to\infty$ and using \eqref{wk-conv1}
		and \eqref{wk-conv2} yield \eqref{LB5-Cheb}.
	\end{proof}

	\begin{remark}
        \begin{enumerate}[(a)]
            \item The theorem applies to any sequence $\{a_j\}_{j=1}^\infty\subset K$ with finitely many limit points all of which appears among the sequence elements. The latter condition can always be achieved by appending the finite number of the limit points to the sequence. For the sequence $\{b_j\}_{j=1}^\infty$ one can choose, for example, $b_j=a_j+i\eps_j$ with $\eps_j\to0$ sufficiently fast so that $\sum_{j=1}^\infty g_K(b_j)<\infty$. In particular, if $a_j=a$ for all $j$ we get a weight function $w$ with a strong zero at $a$, that is, $w(x)\to0$ as $x\to a$ faster than any power function $|x-a|^\al$.

            \item Under an additional assumption that the partial products $\prod_{j=1}^k\frac{x-a_j}{x-b_j}$ are non-negative on $K$ for all large $k$, a minor modification of the proof that relies on \eqref{LB1-Cheb} instead of \eqref{LB3-Cheb}, shows that \eqref{LB5-Cheb} also holds %for weights of the form
                $w(x)=\prod_{j=1}^\infty\big|\frac{x-a_j}{x-b_j}\big|^{1/2}$.

            \item The assumption that $K$ is regular and $\{a_j\}_{j=1}^\infty\subset K$ can be dropped. In this case, the proof remains the same except \eqref{wk-conv3} which becomes
                \begin{align}
                    2S(K,w_k)\exp\biggl[ -\sum_{j=1}^{k}g_K(a_j) \bigg] \le
                    \frac{\|w_k T_{n,w_k}\|_K}{\ca(K)^n} \le \frac{\|w_k T_{n,w}\|_K}{\ca(K)^n}
		        \end{align}
                and yields
                \begin{align}\label{LB6-Cheb}
			         W_{\infty,n}(K,w)\ge 2S(K,w)\exp\biggl[ -\sum_{j=1}^{\infty}g_K(a_j) \bigg], \quad n\ge1.
		        \end{align}

        \end{enumerate}
	\end{remark}
	
	For future use, we introduce the following definition and record a few simple observations.
	
    \begin{definition}\label{rpw-def}
        Let $K\subset\bb R$ be a regular compact set. We call a weight $w$ on $K$ a \emph{rational product weight} if it is not identically zero and is of one of the following two forms:
        \begin{align}\label{rpw}
            w(x) = \bigg|\frac{P_{d_0}(x)}{Q_{d_1}(x)}\bigg|
            \quad\text{or}\quad
            w(x) = \bigg|\frac{P_{d_0}(x)}{Q_{d_1}(x)}\, \prod_{j=1}^\infty\frac{x-a_j}{x-b_j}\bigg|,
        \end{align}
        where $\{a_j\}_{j=1}^\infty\subset K$, $\{b_j\}_{j=1}^\infty\subset\bb C\bs K$ satisfy the assumptions of Theorem~\ref{prod-w-thm}, $P_{d_0}$ is a polynomial of degree $d_0$ with all zeros in $K$, and $Q_{d_1}$ is a polynomial of degree $d_1$ with all zeros outside of $K$.
    \end{definition}

    By \eqref{Sw-rat} and Theorem~\ref{prod-w-thm}, every rational product weight $w$ is in the Szeg\H{o} class on $K$, that is, $S(K,w)>0$ since $S(K,w_1w_2)=S(K,w_1)S(K,w_2)$ for any bounded weights $w_1,w_2$.

    \begin{lemma}\label{prod-rpw}
         A product of finitely many rational product weights is again a rational product weight.
    \end{lemma}
    \begin{proof}
        A product of finitely many polynomials is again a polynomial. Moreover, $m$ pairs of sequences $\{a_{k,j},b_{k,j}\}_{j=1}^\infty$, $k=1,\dots,m$, satisfying the assumptions of Theorem~\ref{prod-w-thm} can be rearranged into a single pair of sequences $\{a_j,b_j\}_{j=1}^\infty$, defined by $a_{(j-1)m+k}=a_{k,j}$, $b_{(j-1)m+k}=b_{k,j}$, $k=1,\dots,m$, $j\ge1$, which again satisfies the assumptions of Theorem~\ref{prod-w-thm} and hence $\prod_{k=1}^m \prod_{j=1}^\infty\big|\frac{x-a_{k,j}}{x-b_{k,j}}\big| = \prod_{j=1}^\infty\big|\frac{x-a_j}{x-b_j}\big|$.
    \end{proof}
	
    %In addition, a minor modification of the proof of Theorem~\ref{prod-w-thm} yields the optimal lower bound for rational product weights:

	\begin{theorem}\label{prod-w-thm2}
		Let $K\subset\bb R$ be a regular compact set and suppose $w$ is a rational product weight on $K$ of the from \eqref{rpw}. Then for all $n>d_1-d_0$,
		\begin{align}\label{LB7-Cheb}
			W_{\infty,n}(K,w)\ge 2S(K,w).
		\end{align}
	\end{theorem}
	\begin{proof}
        In the case of a rational weight $w$, \eqref{LB7-Cheb} follows from \eqref{LB3-Cheb}. In the other case of $w$, the proof requires only a slight modification of the proof of Theorem~\ref{prod-w-thm}.

        Let $w_0=|P_{d_0}/Q_{d_1}|$ and recall that $w_k(x) = \prod_{j=1}^k\big|\frac{x-a_j}{x-b_j}\big|$ converge monotonically and uniformly on $K$. Then $w_0w_k\to w$ monotonically and uniformly on $K$ and so as in \eqref{wk-conv1} and \eqref{wk-conv2}, we get $\|w_0w_kT_{n,w}\|_K\to\|wT_{n,w}\|_K$ and $S(K,w_0w_k)\to S(K,w)$. By \eqref{LB3-Cheb}, we have for all $n>d_1-d_0$,
		\begin{align}
			2S(K,w_0w_k) \le \frac{\|w_0w_kT_{n,w_0w_k}\|_K}{\ca(K)^n} \le \frac{\|w_0w_kT_{n,w}\|_K}{\ca(K)^n}.
		\end{align}
		Then taking the limit as $k\to\infty$ finishes the proof.
	\end{proof}
	
    \begin{remark}
        The assumption that $P_{d_0}$ has zeros on $K$ can be dropped. In this case, instead of \eqref{LB7-Cheb} we get \eqref{LB3-Cheb}, where $\{a_j\}_{j=1}^{d_0}$ are the zeros of $P_{d_0}$ repeated according to their multiplicities.
    \end{remark}

	%%%%%%%%%%%%%%%%%%%%%%%%%%%%%%%%%%%%%%%%%%%%%%%%%%%%%%%%%%%%%%%%%%%%%%%
	\section{Asymptotic Lower Bounds for $L^\infty$ norm}
	%%%%%%%%%%%%%%%%%%%%%%%%%%%%%%%%%%%%%%%%%%%%%%%%%%%%%%%%%%%%%%%%%%%%%%%
	
	In this section, we extend Widom's asymptotic lower bound \eqref{Wid-ALB-Cheb} to Szeg\H{o} class weights with finitely many zeros on $K$. We start with some preparatory results.
	
	\begin{lemma}
		Let $K\subset\bb R$ be a compact set containing an interval $I=[\al,\be]$. Then for all $z=x+iy$ with $x\in(\al,\be)$ the Green function $g_K(z)$ satisfies
		\begin{align}\label{GInt-UB}
			|g_K(z)| \le \frac{|y|}{\sqrt{(\be-x)(x-\al)}}.
		\end{align}
	\end{lemma}
	\begin{proof}
		By the maximum principle for harmonic functions, $g_K(z)\le g_I(z)$ for all $z$ and hence it suffices to prove the inequality for $g_I$. By shifting and scaling, we may assume without loss of generality that $I=[-1,1]$ in which case $g_I(z) = \log|z+\sqrt{z^2-1}| = \Re\log\bigl(z+\sqrt{z^2-1}\,\big)$. Then using the Cauchy--Riemann equations, we obtain
		\begin{align}
			\frac{\pd}{\pd y}g_I(z) = -\Im\frac{d}{dz}\log\bigl(z+\sqrt{z^2-1}\,\big) = -\Im\frac{1}{\sqrt{z^2-1}}.
		\end{align}
		It follows that
		$|\frac{\pd}{\pd y}g_I(z)| \le \frac1{|z+1|^{1/2}|z-1|^{1/2}} \le \frac1{\sqrt{1-x^2}}$.
		Then since $g_I(x)=0$ we have
		\begin{align}
			|g_I(z)| = \bigg|\int_0^y \frac{\pd}{\pd t}g_I(x+it)dt\bigg| \le \frac{|y|}{\sqrt{1-x^2}}.
		\end{align}
	\end{proof}
	
	\begin{theorem}\label{single-zero-w}
		Let $K\subset\bb R$ be a compact set given by a disjoint union of a closed interval $[\al,\be]$ and some other closed set. Let $w$ be a bounded weight of the Szeg\H{o} class on $K$ $($i.e., $S(K,w)>0$$)$, continuous at almost every point of $[\al,\be]$, and bounded away from zero outside arbitrarily small neighborhoods of $x_0\in K$ $($i.e., for each $\eps>0$ there is $L>0$ such that $w(x)\ge L$ for all $x\in K$ with $|x-x_0|>\eps$$)$. Then there exists a weight function $w_0$ of the form \eqref{prod-w} with the zeros and poles satisfying the assumptions of Theorem~\ref{prod-w-thm} and such that for some $C>0$,
		\begin{equation}\label{w0-w-ineq}
			C w_0(x) \leq w(x), \quad x\in K.
		\end{equation}
	\end{theorem}
	\begin{proof}
		Without loss of generality, we may assume that $w\le\frac12$ on $K$ and $x_0$ is not the left end point of $[\al,\be]$. Pick a strictly increasing sequence $\{y_k\}_{k=1}^\infty\subset (\al,\be)$ such that $y_k\to x_0$. Since $w$ is of the Szeg\H{o} class and $\log w\le0$ we have
		\begin{align}\label{ImprRInt}
			\sum_{k=1}^\infty \int_{y_{k}}^{y_{k+1}}|\log w(x)|\, d\mu_K(x) \le
			\int_K |\log w(x)|\, d\mu_K(x) < \infty.
		\end{align}
		
		It is known that the equilibrium measure $\mu_K$ can be regarded as the balayage of the Dirac point mass at infinity onto $K$. This connection together with the properties of balayage measures \cite[Section~2.1]{Tot06} imply that on the interval $[\al,\be]$, the equilibrium measure $d\mu_K$ is absolutely continuous $d\mu_K(x) = w_K(x)dx$ with the density function $w_K(x)$ continuous on $(\al,\be)$ and satisfying,
		\begin{align}\label{wK-sqroot}
			w_K(x)\ge \frac{D}{\sqrt{(x-\al)(\be-x)}}, \quad x\in(\al,\be),
		\end{align}
		for some $D>0$.
		
		By assumption, $w(x)$ is continuous at a.e.\ $x\in[\al,\be]$ and $\log w$ is bounded on each $[y_{k},y_{k+1}]$. It follows that $\log w$ is Riemann--Stieltjes integrable with respect to $d\mu_K$ on each $[y_{k},y_{k+1}]$ and hence $\int_{y_{k}}^{y_{k+1}}|\log w(x)|\, d\mu_K(x)$ can be approximated by the upper Darboux sum. For each $k\ge1$, let $\{x_j\}_{j=n_k}^{n_{k+1}}$ be a partition of $[y_k,y_{k+1}]$ such that
		\begin{align}\label{RSum}
			\bigg|\int_{y_{k}}^{y_{k+1}}|\log w(x)|\,d\mu_K(x) - \sum_{j=n_k}^{n_{k+1}-1} \ell_j \mu_K\big([x_j,x_{j+1}]\big)\bigg| < \frac1{2^k},
		\end{align}
		where $\ell_j=\sup_{x\in[x_j,x_{j+1}]}|\log w(x)|$, $j\ge1$.
		Then by \eqref{ImprRInt} and \eqref{RSum},
		\begin{align}\label{DiscrSz}
			\sum_{j=1}^\infty \ell_j \mu_K\big([x_j,x_{j+1}]\big) < \infty.
		\end{align}
		
		The remainder of the proof is split into two cases.
		
		\textbf{Case I}: Assume that $x_0$ is a boundary point of $[\al,\be]$. Without loss of generality, we may assume $x_0=\be$. Now set
		\begin{align}
			&a_0=x_0,   && b_0=a_0+i, &&r_0=1 \\
			&a_j=x_j,   && b_j=a_j+i(x_{j+1}-x_j), && r_j=\lceil\ell_j/\log 2\rceil, \quad j\ge1,
		\end{align}
		and define the weight function $w_0$ by the infinite product $w_0(x)=\prod_{j=0}^\infty\big|\frac{x-a_j}{x-b_j}\big|^{r_j}$.
		Since by construction $\big|\frac{x-a_j}{x-b_j}\big| \le 1$ on $\bb R$ and $\big|\frac{x-a_j}{x-b_j}\big| \le \frac12$ on $[x_j,x_{j+1}]$
		we have $w_0\le 1$ on $\bb R$ and $w_0 \le 2^{-r_j}$ on $[x_j,x_{j+1}]$, hence
		\begin{align}
			w_0(x) \le e^{-r_j\log2} \le e^{-\ell_j} \le w(x), \quad x\in[x_j,x_{j+1}],\; j\ge1.
		\end{align}
		In addition, $w_0(x_0)=0$ by construction, so we have
		\begin{align}\label{w0-w-loc}
			w_0(x) \le w(x), \quad x\in[y_1,x_0].
		\end{align}
		Since $[\al,\be]$ is a positive distance away from $K\bs[\al,\be]$, \eqref{w0-w-loc} implies that $w_0\le w$ in an (open) neighborhood of $x_0$. Then by assumption there is $L>0$ such that $w\ge L$ outside of that neighborhood of $x_0$ hence \eqref{w0-w-ineq} holds with $C=\min\{1,L\}$.
		
		Since $x_j\to x_0=\be$, it follows from \eqref{wK-sqroot} and \eqref{DiscrSz} that
		\begin{align}\label{DiscrSzEndPt}
			\sum_{j=1}^\infty \ell_j \frac{x_{j+1}-x_j}{\sqrt{\be-x_j}} < \infty.
		\end{align}
		Recalling that $w\le\frac12$ we have $\ell_j\ge|\log w|\ge\log2$ and hence $r_j=\lceil\ell_j/\log2\rceil \le C_1\ell_j$ for all $j$ with some universal constant $C_1>0$. Then using \eqref{GInt-UB}, for some constants $C_2, C_3 >0$, we get
		\begin{align}
			\sum_{j=1}^\infty r_j g_K(b_j) \le C_2\sum_{j=1}^\infty r_j\frac{|b_j-a_j|}{\sqrt{\be-a_j}} \le C_3\sum_{j=1}^\infty \ell_j\frac{x_{j+1}-x_j}{\sqrt{\be-x_j}} < \infty.
		\end{align}
		In addition, since $x_j\in(\al,\be)$ for all $j\ge1$ we also have
		\begin{align}
			\sum_{j=1}^\infty r_j|b_j-a_j| = \sum_{j=1}^\infty r_j(x_{j+1}-x_j) &\le C_1\sum_{j=1}^\infty \ell_j(x_{j+1}-x_j)
			\no \\
			&\le C_1\sqrt{\be-\al}\sum_{j=1}^\infty \ell_j \frac{x_{j+1}-x_j}{\sqrt{\be-x_j}} < \infty.
		\end{align}
		Thus, the assumptions of Theorem~\ref{prod-w-thm} are satisfied for the infinite product $w_0$.

		\textbf{Case II}: Assume that $x_0$ is an interior point of $[\al,\be]$. In this case, the construction in the left vicinity of $x_0$ from the beginning of the proof has to be repeated for the right vicinity of $x_0$. Thus, we pick a strictly decreasing sequence $\{\ti y_k\}_{k=1}^\infty\subset(\al,\be)$ and approximate Riemann--Stieltjes integrals $\int_{\ti y_{k+1}}^{\ti y_k}|\log w(x)|\, d\mu_K(x)$ by the upper Darboux sums. As before this gives us a strictly decreasing sequence $\{\ti x_j\}_{j=1}^\infty$ such that $\ti x_1=\ti y_1$, $\ti x_j\to x_0$
		and $\ti\ell_j:=\sup_{x\in[\ti x_{j+1},\ti x_j]}|\log w(x)|$, $j\ge1$, satisfy
		\begin{align}\label{DiscrSz2}
			\sum_{j=1}^\infty \ti\ell_j \mu_K\big([\ti x_{j+1},\ti x_j]\big) < \infty.
		\end{align}
		Then we set
		\begin{align}
			&a_0=x_0, &&b_0=a_0+i, &&r_0=1, \\
			&a_{2j}=x_j, &&b_{2j}=a_{2j}+i(x_{j+1}-x_j), &&r_{2j}=\lceil\ell_j/\log 2\rceil, \\
			&a_{2j-1}=\ti x_j, &&b_{2j-1}=a_{2j-1}+i(\ti x_j-\ti x_{j+1}), &&r_{2j-1}=\lceil\ti\ell_j/\log 2\rceil, \quad j\ge1,
		\end{align}
		and define the weight function $w_0$ by the infinite product $w_0(x)=\prod_{j=0}^\infty\big|\frac{x-a_j}{x-b_j}\big|^{r_j}$.
		As before the construction yields $w_0\le w$ on $[y_1,x_0]$ and now also on $[x_0,\ti y_1]$. Thus, the inequality holds on a neighborhood of $x_0$ and hence as before we conclude that  \eqref{w0-w-ineq} holds with some constant $C>0$.
		
		Since $\al<x_1<x_j,\ti x_j<\ti x_1<\be$ for all $j$, it follows from \eqref{wK-sqroot} and \eqref{DiscrSz}, \eqref{DiscrSz2} that for some constant $C_1>0$,
		\begin{align}
			\sum_{j=1}^\infty r_j|b_j-a_j|
			&= \sum_{j=1}^\infty r_{2j}(x_{j+1}-x_j) + \sum_{j=1}^\infty r_{2j-1}(\ti x_j-\ti x_{j+1})
			\no \\
			&\le C_1\sum_{j=1}^\infty \ell_j\mu_K\big([x_j,x_{j+1}]\big) +
			C_1\sum_{j=1}^\infty \ti\ell_j\mu_K\big([\ti x_{j+1}-\ti x_j]\big) < \infty.
		\end{align}
		Then using \eqref{GInt-UB} and the above inequality we also get that for some $C_2>0$,
		\begin{align}
			\sum_{j=1}^\infty r_j g_K(b_j) \le C_2\sum_{j=1}^\infty r_j|b_j-a_j| < \infty.
		\end{align}
		Thus, the assumptions of Theorem~\ref{prod-w-thm} are satisfied for the infinite product $w_0$.
	\end{proof}

    A compact set $K\subset\bb R$ is called finite gap if it is a disjoint union of intervals. It is known that finite gaps sets are regular for potential theory.
	
	\begin{corollary}\label{multi-zero-w}
		Let $K\subset\bb R$ be a finite gap compact set and suppose $w$ is a bounded, Szeg\H{o} class weight on $K$ such that $w$ is continuous at a.e.\ point of $K$ and bounded away from zero outside of arbitrarily small neighborhoods of some finite set $X=\{x_1,\dots,x_m\}\subset K$.
		Then, there exists a rational product weight $w_0$ such that $w_0\le w$ on $K$.
%Szeg\H{o} class weight function $w_0$ of the form \eqref{prod-w} with the zeros and poles satisfying the assumptions of Theorem~\ref{prod-w-thm} and such that for some $C>0$ and all $x\in K$,
%		\begin{align}
%			C w_0(x) \le w(x).
%		\end{align}
	\end{corollary}
	\begin{proof}
        Let $U_j$'s be disjoint neighborhoods of $x_j$'s and define $w_j=w$ on $U_j$ and $w_j=1$ on $K\bs U_j$ for $j=1,\dots, m$. In addition, let $w_{m+1}=1$ on $\cup_{j=1}^m U_j$ and $w_{m+1}=w$ on $K\bs(\cup_{j=1}^m U_j)$ so that $w=w_1\cdots w_{m+1}$. Then the weights $w_1,\dots,w_{m+1}$ satisfy the assumptions of Theorem~\ref{single-zero-w} and hence there are constants $C_j>0$ and weights $w_{0,j}$ of the form \eqref{prod-w} satisfying the assumptions of Theorem~\ref{prod-w-thm} such that $C_j w_{0,j} \le w_j$ on $K$ for each $j=1,\dots,m+1$. Hence $w_0:=C_1w_{0,1}\cdots C_{m+1}w_{0,m+1}$ satisfies $w_0 \le w_1\cdots w_{m+1} = w$. Moreover, since each $C_jw_j$ is a rational product weight, by  Lemma~\ref{prod-rpw}, $w_0$ is also a rational product weight.
	\end{proof}

	\begin{theorem}\label{ALB-Cheb-thm}
		Let $K\subset\bb R$ be a finite gap compact set and suppose $w$ is a Riemann integrable weight on $K$ and either
		\begin{enumerate}[\quad$(a)$]
			\item there exists a rational product weight $w_0$ such that $w_0\le w$ on $K$ or
			\item $w$ is bounded away from zero outside of arbitrarily small neighborhoods of some finite set $X=\{x_1,\dots,x_m\}\subset K$.
		\end{enumerate}
		Then,
		\begin{align}\label{ALB-Cheb}
			\liminf_{n\to\infty}W_{\infty,n}(K,w)\geq 2 S(K,w).
		\end{align}
	\end{theorem}
	\begin{proof}
        The inequality is trivial for non-Szeg\H{o} class weights, so assume $S(K,w)>0$.

		(a) Since $w$ is Riemann integrable, it is bounded on $K$ and continuous at almost every point of $K$. Let $f(x)=w_0(x)/w(x)$ for all $x\in K$ where $w(x)\neq0$ and $f(x)=1$ where $w(x)=0$. Then $f\le 1$ on $K$ and $f$ is continuous at a.e.\ point of $K$. It follows that there is a sequence of polynomials $P_j$ with zeros outside of $K$ such that $f\le P_j\le2$ on $K$ and $P_j\to f$ pointwise a.e.\ on $K$. Let $w_j(x)=w_0(x)/P_j(x)$. Then each $w_j$ is a rational product weight satisfying $\frac12 w_0\le w_j\le w$ on $K$ and $w_j\to w$ a.e.\ on $K$. Hence
		\begin{align}
			\|w_j T_{n,w_j}\|_K
			\le \|w_j T_{n,w}\|_K
			\le \|w T_{n,w}\|_K
		\end{align}
		so $W_{\infty,n}(K,w_j) \le W_{\infty,n}(K,w)$.
		
		 By Theorem~\ref{prod-w-thm2}, $\liminf_{n\to\infty}\|w_jT_{n,w_j}\|_K\ge2S(K,w_j)$, hence
		\begin{align}\label{W-Swj}
			\liminf_{n\to\infty}W_{\infty,n}(K,w) \ge \liminf_{n\to\infty}W_{\infty,n}(K,w_j) \ge 2S(K,w_j).
		\end{align}
		Since $w_0$ is in the Szeg\H{o} class and $\frac12 w_0\le w_j$, it follows from the dominated convergence theorem that $S(K,w_j)\to S(K,w)$. Thus, taking $j\to\infty$ in \eqref{W-Swj} yields \eqref{ALB-Cheb}.
		
		(b) follows from (a) since by Corollary~\ref{multi-zero-w} there is a rational product weight $w_0$ such that $w_0\le w$ on $K$.
	\end{proof}

	\begin{theorem}\label{Bern-extension}
		Let $w$ be a Riemann integrable weight on an interval $K=[\al,\be]$ such that either
		\begin{enumerate}[\quad$(a)$]
			\item there exists a rational product weight $w_0$ such that $w_0\le w$ on $K$, or
			\item $w$ is bounded away from zero outside of arbitrarily small neighborhoods of some finite set $X=\{x_1,\dots,x_m\}\subset K$, or
            \item $w$ is not of the Szeg\H{o} class, that is, $S(K,w)=0$.
		\end{enumerate}
		Then,
		\begin{align}\label{Bern-Asympt}
			\lim_{n\to\infty}W_{\infty,n}(K,w) = 2 S(K,w).
		\end{align}
	\end{theorem}
	\begin{proof}
		Since by assumption $w$ is Riemann integrable on $K$, it is continuous at almost every point of $K$ and hence $w(x)$ equals its upper semi-continuous modification
		\begin{align}
		\hat w(x)=\lim_{r\to0^+}\sup_{t\in K,|t-x|<r}w(t)
		\end{align}
		a.e.\ on $K$. Then $S(K,\hat w)=S(K,w)$ and it is also known (see e.g.\ \cite[Lemma~1]{NSZ21}) that $W_{\infty,n}(K,\hat w)=W_{\infty,n}(K,w)$ for all $n$. Thus, by replacing $w$ with $\hat w$ we may assume without loss of generality that $w$ is upper semi-continuous.

		For a bounded, non-negative, upper semi-continuous, Szeg\H{o} class weight $w$ on an interval $K$, Widom showed in \cite[Theorem~11.4]{Wid69} that,
		\begin{align}\label{Wid-UB}
			\limsup_{n\to\infty}W_{\infty,n}(K,w) \le 2 S(K,w).
		\end{align}
		This asymptotic upper bound also holds for non-Szeg\H{o} class weights $w$. Indeed, let $w_\eps=w+\eps$, where $\eps>0$ is a constant. Then the weight $w_\eps$ is upper semi-continuous and bounded away from zero hence $w$ is of the Szeg\H{o} class. Thus, \eqref{Wid-UB} holds for $w_\eps$. Since $w\le w_\eps$ we have $W_{\infty,n}(K,w) \le W_{\infty,n}(K,w_\eps)$ for all $n$ and hence $\limsup_{n\to\infty}W_{\infty,n}(K,w) \le 2 S(K,w_\eps)$ for each $\eps>0$. Taking $\eps\to0$ then yields $S(K,w_\eps)\to S(K,w)$ hence \eqref{Wid-UB} holds for all upper semi-continuous weights.

		The corresponding asymptotic lower bound,
        \begin{align}
            \liminf_{n\to\infty}W_{\infty,n}(K,w) \ge 2 S(K,w),
		\end{align}
		trivially holds for non-Szeg\H{o} class weights and follows from Theorem~\ref{ALB-Cheb-thm} for Szeg\H{o} class weights. The two matching asymptotic bounds then give \eqref{Bern-Asympt}.
	\end{proof}
	
    \begin{remark}
		The assumption (b) in Theorems~\ref{ALB-Cheb-thm} and~\ref{Bern-extension} is satisfied for $w=w_0w_1$, where $w_1$ is a Riemann integrable weight bounded away from zero on $K$ and $w_0$ is an arbitrary continuous weight with at most finitely many zeros on $K$. The assumption (a) in Theorems~\ref{ALB-Cheb-thm} and~\ref{Bern-extension} is satisfied for any rational product weight $w$ (Definition~\ref{rpw-def}), in particular, (a) holds for some continuous, Szeg\H{o} class weights with infinitely many zeros.
    \end{remark}

	\section{Lower bounds for $L^2$ norm}
	
	In this section we derive lower bounds for $L^2$ norms of monic polynomials. We assume that $w$ is a weight function on $K$ with $w\in L^1(K,d\mu_K)$, $\|w\|_{L^1(K,d\mu_K)}>0$ and define $d\mu=wd\mu_K$. As is well known the $n$-th degree monic orthogonal polynomial $P_n$ minimizes the $L^2(K,d\mu)$ norm among all monic polynomials of degree $n$. The Widom factors for the $L^2(K,d\mu)$ norm are defined in \eqref{l2 wid}.
	
	We start with a preparatory result on the invariance of harmonic measures under the universal covering map. While there are several versions of Theorem~\ref{pre1} in the literature, neither is stated in a suitable form or sufficient generality for our purposes. In \cite[Section 2.4]{Fi83}, regularity of the set is assumed and in \cite[Section 3.7B]{Hasumi}, Martin boundary is considered instead of Euclidean boundary. %So we give a proof the result below, which has the most general and suitable form for problems involving extremal polynomials on the real line.
We use $\omega_{\Omega_K}(a;\cdot)$ to denote the harmonic measure for $\Omega_K$ at $a\in\Omega_K$.

	\begin{theorem}\label{pre1}
		Let $K$ be a non-polar compact subset of $\mathbb{C}$. Let $\x:\bb D\rightarrow\Omega_K$ be the universal covering map uniquely fixed by the conditions $\x(0)=\infty$ and $\lim_{z\rightarrow 0} z \x(z)>0$. Let $\x^*$ denote the non-tangential limit of $\x$ on $\partial \bb{D}$. Let $a\in\Omega_K$ and $z\in \bb D$ such that $\x(z)=a$. Then for any $f\in L^1(K,d\mu_K)$, we have
		\begin{align}\label{K2dD12}
			\int_K f(t)d \omega_{\Omega_K}(a;t) = \int_0^{2\pi}\Re\frac{e^{i\te}+z}{e^{i\te}-z}f(\x^*(e^{i\te}))\frac{d\te}{2\pi},
		\end{align}
		and in particular
			\begin{align}\label{K2dD11}
			\int_K f(t)d\mu_K(t) = \int_0^{2\pi}f(\x^*(e^{i\te}))\frac{d\te}{2\pi}.
		\end{align}
	\end{theorem}
	\begin{proof}
		Let $d\mu_{z;\partial \bb{D}}$ be $\Re\frac{e^{i\te}+z}{e^{i\te}-z}\frac{d\te}{2\pi}$ and $d\mu_{\partial \bb D}$ be the normalized Lebesgue measure on $\partial \bb D$. Since $\x$ is a universal covering map, the non-tangential limit $\x^*(e^{i\theta})$ exists a.e. $d\theta$ and $\x^*(e^{i\theta})\in\partial \Omega_K$, a.e. $d\theta$, see e.g. \cite[Chapter~7]{Nev70}.
		
		Let $\nu$ be the push-forward measure of   $d\mu_{z;\partial \bb{D}}$ through $\x^*$. Thus, for every $u\in L^1(d\nu)$, we have
		\begin{align}\label{preser1}
			\int u\,d\nu = \int_0^{2\pi}u(\x^*(e^{i\te})) d\mu_{z;\partial \bb{D}}(\theta).
		\end{align}
		
		Let $E$ be the set of irregular boundary points and $F$ be the set of regular boundary points of the domain $ \Omega_K$. Then $E$ is a $F_\sigma$ polar set, see e.g. \cite[Theorem 4.2.5]{Ran95}. In view of
		\cite[Theorem~VIII.44]{Tsuji}, $ \mu_{\partial \bb{D}} \left({(\x^*)}^{-1}(E)\right)=0$ and thus $ \mu_{\partial \bb{D}} \left({(\x^*)}^{-1}(F)\right)=1$.
		Let $u$ be a continuous real-valued function on $\partial\Omega_K$ and $H_u$ denote the solution of the Dirichlet problem in $\Omega_K$ for the boundary function $u$. Then, for any $z_0\in F$,
		\begin{equation}
			\lim_{z\rightarrow z_0, z\in \Omega_K} H_u(z)= u(z_0),
		\end{equation}
		see e.g. \cite[Theorem~2.1, Appendix~A.2]{ST97}.
		Therefore,
		\begin{equation}
			\lim_{r\uparrow 1} H_u(\x(r e^{i\theta}))= u(\x^*(e^{i\theta}))\,\, \mathrm{a.e.} \,\, d\theta.
		\end{equation}
		Hence, $H_u(\x)$ is harmonic and bounded in $\bb D$ with non-tangential boundary values  $u(\x^*(e^{i\theta}))$ for $d\theta$ a.e.
		Therefore,
		\begin{equation}\label{preser2}
			\int_0^{2\pi}u(\x^*(e^{i\te}))d\mu_{z;\partial \bb{D}}(\theta)= H_u(\x(z))=H_u(a)=\int u(t)\,d \omega_{\Omega_K}(a;t)
		\end{equation}
		Clearly, \eqref{preser2} holds for any complex valued continuous function $u$ on $\partial \Omega_K$.
		Combining \eqref{preser1} and \eqref{preser2}, we see that, for any continuous $u:\partial \Omega_K \rightarrow \bb C$,
		\begin{equation}
			\int u\, d\nu = \int u(t)\,d \omega_{\Omega_K}(a;t).
		\end{equation}
		Since $\mathrm{supp}(\nu)\subset \partial \Omega_K$, this implies that $d\,\nu=d\, \omega_{\Omega_K}(a;\cdot)$. So \eqref{preser1} implies \eqref{K2dD12}. The equation \eqref{K2dD11} is just a special case of \eqref{K2dD12} for $z=0$ and $a=\infty$.
	\end{proof}

	\begin{theorem}\label{ALB-OP-Thm}
		Let $K\subset\bb R$ be a non-polar compact set and $w$ be a weight function on $K$ such that $w\in L^1(K,d\mu_K)$ and $S(K,w)>0$. Then,
		\begin{align}\label{ALB-OP}
			\liminf_{n\to\infty}[W_{2,n}(K,w)]^2\geq 2 S(K,w).
		\end{align}
		Suppose, in addition that, $K$ is regular and the weight $w$ is a rational function,
		\begin{align}\label{w-rat}
			w(x)=c \frac{\prod_{j=1}^{d_0}(x-a_j)} {\prod_{j=1}^{d_1}(x-b_j)},
		\end{align}
		which is bounded on $K$ and let $d=d_1-d_0$. Then for all $n>d/2$,
		\begin{align}\label{LB1-OP}
			[W_{2,n}(K,w)]^2 \geq \frac{2S(K,w)}{1 + \sqrt{1-\exp\bigl[-2\sum_{j=1}^{d_0}g_K(a_j)\big]}}.
		\end{align}
		Furthermore, assume that $a_j\in K$ for each $j$. Then for all $n>d/2$,
		\begin{align}\label{LB2-OP}
			[W_{2,n}(K,w)]^2 \geq  2S(K,w).
		\end{align}		
		The equality in \eqref{LB2-OP} is attained for some $n>d/2$ if and only if there is a polynomial $Q_n$ of degree $n$ such that its zeros are disjoint from $\{b_j\}_{j=1}^{d_1}$ and
		\begin{equation}\label{ort rat inv1}
			K=(w Q_n^2)^{-1}([0,1]).
		\end{equation}
		In that case, $P_n=T_{n,\sqrt{w}}$ and $Q_n=r P_{n}$, where $r$ is a non-zero real number.
	\end{theorem}
	
	\begin{proof}
		We will proceed by transferring the problem to the unit disk using the uniformization (see \cite[Chapter~2]{Fi83}, \cite[Chapters~13--16]{Mar19}).
		Let $\x:\bb D\rightarrow \overline{\bb C}\bs K$ be the universal covering map uniquely fixed by the conditions $\x(0)=\infty$ and $\lim_{z\rightarrow 0} z \x(z)>0$. Let $B(z)$ be the analytic function in $\mathbb{D}$ determined by
		\begin{align}\label{Bat0}
			|B(z)| = e^{-g_K(\x(z))} \;\text{ and }\;
			\lim_{z\to0}\x(z)B(z) = \ca(K)
		\end{align}	
		and $B(z,z_0)$ be the analytic function in $\bb D$ determined by
		\begin{align}\label{Bz0at0}
			|B(z,z_0)|= e^{-g_K(\x(z),\x(z_0))} \;\text{ and }\; B(0,z_0)>0.
		\end{align}
		The covering map has radial limits $\x^*(e^{i\te})\in K$ a.e.\ on $\pd\bb D$ and the equilibrium measure is preserved under the covering map (see Theorem~\ref{pre1}),
		\begin{align}\label{K2dD}
			\int_K f(t)d\mu_K(t) = \int_0^{2\pi}f(\x^*(e^{i\te}))\frac{d\te}{2\pi}
			\;\text{ for all $f\in L^1(K,d\mu_K)$.}
		\end{align}

		By assumption $w\in L^1(K,\mu_K)$ and $S(K,w)>0$ hence $\log(w(\cdot))\in L^1(K,d\mu_K)$ and so $\log(w\circ\x^*) \in L^1(\pd\bb D,\frac{d\te}{2\pi})$. It follows that there is an analytic function $R(z)$ on $\bb D$ with $R(0)>0$ and the boundary values $|R(e^{i\te})|^2=w(\x^*(e^{i\te}))$,
		\begin{align}
			R(z) = \exp\left(\frac12 \int_0^{2\pi} \frac{e^{i\te}+z}{e^{i\te}-z} \log[w(\x^*(e^{i\te}))]\frac{d\te}{2\pi} \right).
		\end{align}
		We note that $R\in H^2(\bb D)$ and by \eqref{K2dD} and \eqref{S-def},
        \begin{align}\label{R2at0}
            R(0)=\sqrt{S(K,w)}.
        \end{align}
		
		Let $P_n$ denote the $n$-th monic orthogonal polynomial in $L^2(d\mu)$ where $d\mu=wd\mu_K$. Since $B^n(z)P_n(\x(z))$ is a bounded analytic function on $\bb D$ we have $B^nP_n(\x(\cdot))R\in H^2(\bb D)$ and hence
		\begin{align}
			\int_{0}^{2\pi} B^n(e^{i\te})P_n(\x^*(e^{i\te}))R(e^{i\te}) \frac{d\theta}{2\pi}
			= \ca(K)^n R(0)
			= \ca(K)^n \sqrt{S(K,w)}.
		\end{align}
		Let $U(e^{i\te})=R(e^{i\te})/|R(e^{i\te})|$ so that $R(e^{i\te})=\sqrt{w(\x^*(e^{i\te})}U(e^{i\te})$. Since the orthogonal polynomials on the real line are necessarily real we have
		\begin{align}\label{OP-intg}
			\int_{0}^{2\pi} P_n(\x^*(e^{i\te}))\sqrt{w(\x^*(e^{i\te})} \Re[U(e^{i\te})B^n(e^{i\te})] \frac{d\theta}{2\pi}
			=\sqrt{S(K,w)}\ca(K)^n.
		\end{align}
		Then by Cauchy--Schwarz inequality,
		\begin{align}\label{CS1}
			S(K,w)\ca(K)^{2n} &\le
			\int_0^{2\pi} |P_n(\x^*(e^{i\te}))|^2 w(\x^*(e^{i\te})) \frac{d\theta}{2\pi}
			\int_0^{2\pi} \big|\Re[U(e^{i\te})B^n(e^{i\te})]\big|^2 \frac{d\theta}{2\pi}.
		\end{align}
		Since $|U(e^{i\te})B^n(e^{i\te})|=1$, we have
		\begin{equation}\label{Re-Re}
		2\big|\Re[U(e^{i\te})B^n(e^{i\te})]\big|^2 = 1+\Re[U^2(e^{i\te})B^{2n}(e^{i\te})]
		\end{equation}
		and hence
		\begin{align}\label{CS2}
			2S(K,w) \le \big[W_{2,n}(K,w)\big]^2  \left(1+\Re\int_0^{2\pi}U^2(e^{i\te})B^{2n}(e^{i\te}) \frac{d\theta}{2\pi} \right) .
		\end{align}
		As noted in \cite[Lemma~8.5]{CSZ11}, the functions $\{B^m\}_{m\in\bb Z}$ are orthonormal in $L^2(\pd\bb D,\frac{d\te}{2\pi})$ by Cauchy's theorem, hence it follows from Bessel's inequality that
		\begin{align}\label{Bes-asymp}
			\lim_{n\to\infty} \int_0^{2\pi}U^2(e^{i\te})B^{2n}(e^{i\te}) \frac{d\te}{2\pi}=0.
		\end{align}
		Thus, the asymptotic lower bound \eqref{ALB-OP} follows from \eqref{CS2} and \eqref{Bes-asymp}.
		
		\smallskip
		
		Next, we turn to the non-asymptotic lower bound \eqref{LB1-OP} for bounded rational weights $w$. Without loss of generality, we assume that the numerator and denominator of $w$ in \eqref{w-rat} have no common zeros. By reordering the zeros if necessary, we also assume that
$\{a_j\}_{j=1}^{d_0'}$, $0\leq d_0' \leq d_0$, are the zeros of $w$ (repeated according to their multiplicities) that  lie outside of $K$.

		Pick $z_j\in\bb D$ such that $\x(z_j)=a_j$ for $j=1,\dots,d_0'$ and $p_j\in\bb D$ such that $\x(p_j)=b_j$ for $j=1,\dots,d_1$. Consider the function
        \begin{align}\label{L-def}
            L(z)=w(\x(z))B^{d_0}(z)\prod_{j=1}^{d_1} B(z,p_j) \text{  on } \bb D.
        \end{align}
By construction, $L(z)$ has removable singularities at the poles of $w(\x(z))$ and we extend $L(z)$ to an analytic function on $\bb D$. Since by assumption $w$ is bounded on $K$, it is also bounded in a neighborhood of $K$ and hence $L(z)$ is bounded on $\bb D$. Then by the inner-outer factorization,
        \begin{align}\label{L-inn-out}
            L(z)=\vartheta(z)B^{d_1}(z)\prod_{j=1}^{d_0'}B(z,z_j)R^2(z),
        \end{align}
where $\vartheta(z)$ is some inner function, in particular $|\vartheta(z)|\le1$ on $\bb D$. Then, using \eqref{L-inn-out}, \eqref{L-def}, \eqref{Bat0}, \eqref{Bz0at0}, \eqref{R2at0}, and \eqref{Sw-rat} we obtain
		\begin{align}
			|\vartheta(0)| &= \lim_{z\to0} \left| \frac{w(\x(z))B^{d_0}(z)\prod_{j=1}^{d_1}B(z,p_j)} {B^{d_1}(z)\prod_{j=1}^{d_0'}B(z,z_j)R^2(z)} \right|
\no \\
&= \frac{|c|\ca(K)^{d_0} \exp\left[-\sum_{j=1}^{d_1}g_K(b_j)\right]} {\ca(K)^{d_1}\exp\left[-\sum_{j=1}^{d_0'}g_K(a_j)\right]S(K,w)} = 1
		\end{align}
which implies, by the maximum modulus principle, that $\vartheta(z)$ is a unimodular constant on $\bb D$. Thus, for some constant $\psi\in \partial \bb D$,
		\begin{align}
			R^2(z)=\psi w(\x(z)) B^{-d}(z) \frac{\prod_{j=1}^{d_1}B(z,p_j)}{\prod_{j=1}^{d_0'}B(z,z_j)},
		\end{align}
		and hence
		\begin{align}
			U^2(z)=\psi B^{-d}(z) \frac{\prod_{j=1}^{d_1}B(z,p_j)}{\prod_{j=1}^{d_0'}B(z,z_j)}.
		\end{align}
		
		It follows from Cauchy's theorem that in $L^2(\pd\bb D,\frac{d\te}{2\pi})$ the constant function $\cf$ is orthogonal to $\prod_{j=1}^{d_1}B(\cdot,p_j)B^m(\cdot)$ for any $m\in\bb N$ and both have unit norm. Then by Bessel's inequality, we have for $m\in\bb N$,
		\begin{align}
			&\bigg|\bigg< \prod_{j=1}^{d_1}B(\cdot,p_j)B^m(\cdot), \prod_{j=1}^{d_0'}B(\cdot,z_j) \bigg>\bigg|^2
			\leq
			\bigg\| \prod_{j=1}^{d_0'}B(\cdot,z_j) \bigg\|^2 -
			\bigg|\bigg< \cf, \prod_{j=1}^{d_0'}B(\cdot,z_j) \bigg>\bigg|^2
			\\ \no
			&\qquad = 1 - \bigg|\prod_{j=1}^{d_0'}B(0,z_j)\bigg|^2 = 1-\exp\biggl[-2\sum_{j=1}^{d_0'}g_K(\x(z_j))\bigg] =
			1-\exp\biggl[-2\sum_{j=1}^{d_0}g_K(a_j)\bigg].
		\end{align}
		Hence for each $n\in\bb N$ with $n>d/2$ we have
		\begin{align}\label{Bes-ineq}
			&\Re\int_0^{2\pi}U^2(e^{i\te})B^{2n}(e^{i\te}) \frac{d\te}{2\pi}
			\leq \bigg|\int_0^{2\pi}U^2(e^{i\te})B^{2n}(e^{i\te}) \frac{d\te}{2\pi}\bigg|
			\\ \no
			&\qquad = \bigg|\bigg< \prod_{j=1}^{d_0'}B(\cdot,z_j), \prod_{j=1}^{d_1}B(\cdot,p_j)B^{2n-d}(\cdot) \bigg>\bigg|
			\le \bigg(1-\exp\biggl[-2\sum_{j=1}^{d_0}g_K(a_j)\bigg]\bigg)^{1/2}.
		\end{align}
		Combining \eqref{Bes-ineq} with \eqref{CS2} we obtain
		\begin{align}
			2S(K,w) \leq \big[W_{2,n}(K,w)\big]^2
			\bigg(1+\bigg(1-\exp\biggl[-2\sum_{j=1}^{d_0}g_K(a_j)\bigg]\bigg)^{1/2}\bigg)
		\end{align}
		which yields the desired lower bound \eqref{LB1-OP}.

        When $a_j\in K$ for all $j$, we have $g_K(a_j)=0$ for all $j$ and hence \eqref{LB2-OP} follows from \eqref{LB1-OP}.
		
		The only remaining part is the equivalence of \eqref{ort rat inv1} with the equality in \eqref{LB2-OP}.
		First, let us assume that the equality in \eqref{LB2-OP} holds. Since by assumption $K$ is regular and $a_j\in K$ for all $j$, it follows from \eqref{Bes-ineq} that
		\begin{equation}\label{Re bound 0}
		\Re\int_0^{2\pi}U^2(e^{i\te})B^{2n}(e^{i\te}) \frac{d\te}{2\pi}\leq 0.
		\end{equation}
		Hence, we have equality in both \eqref{CS2} and \eqref{CS1} and also
		\begin{equation}\label{Re eq 0}
			\Re\int_0^{2\pi}U^2(e^{i\te})B^{2n}(e^{i\te}) \frac{d\te}{2\pi}= 0.
		\end{equation}
		Equality in the Cauchy--Schwarz inequality \eqref{CS1} and the identity \eqref{Re-Re} yield
		\begin{equation}\label{CS3}
\frac{|P_n(\x^*(e^{i\te}))|^2 w(\x^*(e^{i\te}))}{\|P_n\sqrt{w}\|^2_{L^2(d\mu_K)}}=1+\Re (U^2(e^{i\te})B^{2n}(e^{i\te})) \mbox{ for a.e. } \theta.
		\end{equation}
Since $1+\Re (U^2(e^{i\te})B^{2n}(e^{i\te}))\leq 2$ for a.e.\ $\theta$, combining \eqref{CS3} with the equality in \eqref{LB2-OP}, we see that
\begin{equation}\label{L-infty ineq1}
    \|P_n(\x^*(\cdot))\sqrt{w(\x^*(\cdot))}\|^2_{L^\infty(d\theta)}\leq 4 S(K,w) \ca(K)^{2n}.
\end{equation}	
The inequality \eqref{L-infty ineq1} implies
\begin{equation}\label{L-infty-ineq2}
    \|P_n\sqrt{w}\|_{L^\infty(d\mu_K)}\leq 2 S(K,\sqrt{w}) \ca(K)^{n}.
\end{equation}
In view of the regularity assumption on $K$, we have  $\mathrm{supp}(\mu_K)=K$  (see \cite[Theorem~4.2.3]{Ran95} and \cite[Corollary~5.5.12]{Sim11}). Since $P_n\sqrt{w}$ is continuous on $K$, it follows from \eqref{L-infty-ineq2} that
\begin{equation}
    \|P_n\sqrt{w}\|_{K}\leq 2 S(K,\sqrt{w}) \ca(K)^{n}.
\end{equation}
Thus, $W_{\infty,n}(K,\sqrt{w})\leq 2 S(K,\sqrt{w})$. It was shown in \eqref{LB2-Cheb} that  $W_{\infty,n}(K,\sqrt{w})\geq 2 S(K,\sqrt{w})$. Thus, it follows from Theorem~\ref{LB-Thm} that $K=(wQ_n^2)^{-1}([0,1])$ where $Q_n$ is a polynomial of degree $n$ whose zeros are disjoint from $\{b_j\}_{j=1}^{d_1}$. So \eqref{ort rat inv1} holds.

Conversely, let us assume that \eqref{ort rat inv1} holds.
Let $F_n:=2wQ_n^2-1$. Then
\begin{equation}\label{ort rat inv2}
	K= (w Q_n^2)^{-1}([0,1])= F_n^{-1}([-1,1]).
\end{equation}

In view of Theorem~\ref{LB-Thm},
\begin{equation}\label{Qn}
    Q_n=r T_{n,\sqrt{w}}
\end{equation}
where $|r|=\frac{1}{t_n(K,\sqrt{w})}$. Besides, we have equality in \eqref{cap-g-ineq} and \eqref{tn2capL} which implies that
\begin{equation}\label{great eq}
			\frac{1}{|r|}=t_n(K,\sqrt{w}) = 2 |c|^{1/2} \ca(K)^{n-d/2} \exp\biggl[-\frac12\sum_{j=1}^{d_1}g_{L_n}(b_j)\bigg],
\end{equation}
We argue as in \cite[Theorem~2.10]{ELY24}. Let $J(u)=\frac{1}{2}(u+\frac{1}{u})$ be the Joukowsky map which maps $\bb D$ onto $\overline{\bb{C}}\setminus [-1,1]$ conformally. Then $J^{-1}(F_n(\x(\cdot)))$ is a bounded analytic function on $\bb D$ with the boundary values $|J^{-1}(F_n(\x(e^{i\te})))|=1$ a.e.\ $d\te$, hence by the inner-outer factorization $J^{-1}(F_n(\x(z)))=\vartheta(z) C(z)$, where $\vartheta(z)$ is a singular inner function and
\begin{equation}\label{super great}
	C(z)=B^{2n-d}(z) \prod_{j=1}^{d_1}B(z,p_j) \mbox{ on } \bb{D}.
\end{equation}
In view of \eqref{great eq}, \eqref{super great}, \eqref{Bat0}, \eqref{Bz0at0}, we have
\begin{equation}
\lim_{z\rightarrow 0}	\left|\frac{J^{-1}(F_n(\x(z)))}{C(z)}\right|=1,
\end{equation}
hence $\vartheta$ is constant. Then $J^{-1}(F_n(\x(\cdot)))=\phi C$ for some constant $\phi\in \bb \partial \bb D$ and hence
$F_n(\x(\cdot))= \frac{1}{2}(\phi C+\frac{1}{\phi C})$. Since $\int_{0}^{2\pi} C(e^{i\theta}) \frac{d\theta}{2\pi}=0$ and $|C(e^{i\theta})|=1$, a.e. $\theta$, for any $\phi$, this implies that
\begin{equation}
	\int_{0}^{2\pi} [2w(\x^*(e^{i\te})) Q_n^2(\x^*(e^{i\te})) - 1] \frac{d\theta}{2\pi}= 	\int_{0}^{2\pi} F_n(\x^*(e^{i\te})) \frac{d\theta}{2\pi}= 0.
\end{equation}
Hence
\begin{equation}\label{int eq}
\int w Q_n^2 d\mu_K = \int_{0}^{2\pi} w(\x^*(e^{i\te})) Q_n^2(\x^*(e^{i\te}))\frac{d\theta}{2\pi}=\frac{1}{2}.
\end{equation}
Combining \eqref{Qn}, \eqref{int eq} and \eqref{ort rat inv2}, we get
\begin{equation}\label{L2-Linf-equal}
	\int w T_{n,\sqrt{w}}^2 d\mu_K= \frac{1}{2r^2}= \frac{\|\sqrt{w}T_{n,\sqrt{w}}\|^2_K}{2}={2}S(K,w)\ca(K)^{2n}.
\end{equation}
It follows from \eqref{L2-Linf-equal} that
\begin{equation}
[W_{2,n}(K,w)]^2\leq {2}S(K,w).
\end{equation}
Combined with \eqref{LB2-OP} this yields $[W_{2,n}(K,w)]^2= {2}S(K,w)$. Moreover, since the $L^2(d\mu)$ extremal polynomial $P_n$ is unique, it follows from \eqref{L2-Linf-equal} that $P_n=T_{n,\sqrt{w}}$ and hence $Q_n=rP_n$ by \eqref{Qn}.
	\end{proof}
	
As in the case of Chebyshev polynomials, we can extend the non-asymptotic lower bound \eqref{LB2-OP} to rational product weights.

    \begin{theorem}\label{LB-OP-Thm}
		Let $K\subset\bb R$ be a regular compact set and suppose $w$ is a rational product weight on $K$ of the from \eqref{rpw}. Then \eqref{LB2-OP} holds. In addition, if the assumption that $P_{d_0}$ has zeros on $K$ is dropped in Definition~\ref{rpw-def}, then \eqref{LB1-OP} holds, where $\{a_j\}_{j=1}^{d_0}$ are the zeros of $P_{d_0}$ repeated according to their multiplicities.
    \end{theorem}
    \begin{proof}
		The proof is essentially the same as in Theorem~\ref{prod-w-thm2}. It suffices to consider only the second part where the polynomial $P_{d_0}$ from Definition~\ref{rpw-def} is not assumed to have zeros on $K$.

		Let $w_0=|P_{d_0}/Q_{d_1}|$. If $w=w_0$, the result is contained in Theorem~\ref{ALB-OP-Thm}. So assume $w(x)=w_0(x)\prod_{j=1}^\infty\big|\frac{x-a_j}{x-b_j}\big|$. Let $w_k(x) = \prod_{j=1}^k\big|\frac{x-a_j}{x-b_j}\big|$ and recall that $w_0w_k\to w$ monotonically and uniformly on $K$. Then, as in \eqref{wk-conv1} and \eqref{wk-conv2}, we get
    \begin{align}
    \begin{split}\label{wk-conv-OP}
        \|\sqrt{w_0w_k}P_{n,w}\|_{L^2(d\mu_K)} &\to \|\sqrt{w}P_{n,w}\|_{L^2(d\mu_K)} = W_{2,n}(K,w)\ca(K)^n, \\
        S(K,w_0w_k) &\to S(K,w),
    \end{split}
    \end{align}
    where $P_{n,w}$ is the $n$-th monic orthogonal polynomial in $L^2(wd\mu_K)$. By \eqref{LB1-OP}, we have for all $n>d_1-d_0$,
		\begin{align}
			&\frac{2S(K,w_0w_k)} {1+\sqrt{1-\exp\bigl[-2\sum_{j=1}^{d_0}g_K(a_j)\big]}} \le [W_{2,n}(K,w_0w_k)]^2 \no\\
            &\qquad = \frac{\|\sqrt{w_0w_k}P_{n,w_0w_k}\|_{L^2(d\mu_K)}^2}{\ca(K)^{2n}}
            \le \frac{\|\sqrt{w_0w_k}P_{n,w}\|_{L^2(d\mu_K)}^2}{\ca(K)^{2n}}.
		\end{align}
		Then taking the limit as $k\to\infty$ and using \eqref{wk-conv-OP} finishes the proof.
    \end{proof}

	\begin{remark}\label{kgamma}
		The following construction  can be found in \cite{Gonc14}. Let $\gamma_0 = 1$ and $\gamma = (\gamma_n)_{n=1}^{\infty}$ be a sequence satisfying $0 < \gamma_n < 1/4$ for all $n \in \mathbb{N}$ and \[ \sum_{n=1}^{\infty} 2^{-n} \log (1/\gamma_n) < \infty. \]
		Let $(f_n)_{n=1}^{\infty}$ by \[ f_1(z) := \frac{2z(z - 1)}{\gamma_1} + 1 \quad \text{and} \quad f_n(z) := \frac{z^2}{2\gamma_n} + 1 - \frac{1}{2\gamma_n} \] for $n > 1$.
		
		Let $E_0=[0,1]$ and $E_n=(f_n \circ \ldots \circ f_1)^{-1} ([-1,1])$. Then $K(\gamma) := \cap_{s=0}^{\infty} E_s$ is a non-polar Cantor set in $[0, 1]$ where $\{0, 1\} \subset K(\gamma)$.
		It was shown in \cite[Proposition~3.1]{GonHat15} that
		\begin{equation}\label{Can1}
			W_{\infty,2^n}(K(\gamma), 1)= \frac{1}{2} \exp\left[ 2^n \sum_{s=n+1}^\infty 2^{-s}\log{(1/\gamma_s)} \right]
		\end{equation}	
		and it was proved in \cite[Eq.~(5.3)]{AlpGon16} that
		\begin{equation}\label{Can2}
			W_{2,2^n}(K(\gamma),1)= \frac{\sqrt{1-2\gamma_{n+1}}}{2\exp(\sum_{k=n+1}^\infty 2^{n-k} \log{\gamma_k})}
		\end{equation}
		
		Assume that $\gamma_n\rightarrow 1/4$. Then it follows from \eqref{Can1} and \eqref{Can2} that
\begin{align}
[W_{2,2^n}(K(\gamma),1)]^2 \rightarrow 2
\;\text{ and }\;
W_{\infty,2^n}(K(\gamma), 1)\rightarrow 2.
\end{align}
Since $[W_{2,s}(K(\gamma),1)]^2\geq 2$ (\cite[Theorem~3.1]{AZ20a}) and $W_{\infty,s}(K(\gamma), 1)\geq 2$ (\cite[Theorem~2]{Sch08}),
we have
\begin{align}\label{ALB-satur}
\liminf_{s\rightarrow\infty} [W_{2,s}(K(\gamma),1)]^2=2
\;\text{ and }\;
\liminf_{s\rightarrow\infty} W_{\infty,s}(K(\gamma), 1)= 2.
\end{align}
Hence, such a set $K(\gamma)$ is an example of a Cantor set which realizes the equality in both \eqref{ALB-Cheb} and \eqref{ALB-OP} for $w\equiv 1$. In addition, depending on how fast $\gamma_n$ converges to $1/4$ (see \cite[Section~7]{AlpGon17} for the precise statement), we can find a Cantor set $K(\gamma)$ which is a Parreau--Widom set or a non-Parreau--Widom set with this property.
	\end{remark}
	
%%%%%%%%%%%%%%%%%%%%%%%%%%%%%%%%%%%%%%%%%%%%%%%%%%%%%%%%%%%%%%%%%%%%%
	
	%%%%%%%%%%%%%%%%%%%%%%%%%%%%%%%%%%%%%%%%%%%%%%%%%%%%%%%%%%%%%%%%%%%%%
	
%%%%%%%%%%%%%%%%%%%%%%%%%%%%%%%%%%%%%%%%%%%%%%%%%%%%%%%%%%%%%%%%%%%%%
\end{document}